\documentclass[a4paper,12pt,reqno]{amsart}
\usepackage{amsmath,amsthm,amssymb}
\usepackage{graphicx, color}
\usepackage[numbers,sort&compress]{natbib}
\nonstopmode \numberwithin{equation}{section}

\newtheorem{theorem}{Theorem}[section]

\newtheorem{remark}{Remark}[section]

\allowdisplaybreaks

\allowdisplaybreaks

\begin{document}

\title[$(p,q)$-Whittaker function and its properties ]{$(p,q)$-Whittaker function and associated properties and formulas}

\author[G. Rahman, S. Mubeen,  K. S. Nisar and J. Choi ]{ Gauhar Rahman, Shahid Mubeen, Kottakkaran Sooppy  Nisar\\
and Junesang Choi*}

\address{Gauhar Rahman:    Department of Mathematics, International Islamic  University, Islamabad, Pakistan}
\email{gauhar55uom@gmail.com}

\address{Shahid Mubeen: Department of Mathematics, University of Sargodha,   Sargodha, Pakistan}
\email{smjhanda@gmail.com}

\address{Kottakkaran Sooppy  Nisar:    Department of Mathematics, College of Arts and Science-Wadi Aldawaser, 11991,
Prince Sattam bin Abdulaziz University, Alkharj, Kingdom of Saudi Arabia}
\email{n.sooppy@psau.edu.sa; ksnisar1@gmail.com}

\address{Junesang Choi: Department of Mathematics,
Dongguk University, Gyeongju
38066, Republic of Korea}
\email{junesang@mail.dongguk.ac.kr}



\thanks{*Corresponding author}

\begin{abstract}
Recently, various extensions  and variants of Bessel functions of several kinds have been presented.
Among them, the $(p,q)$-confluent hypergeometric function $\Phi_{p,q}$  has been introduced and investigated.
Here, we aim to    introduce an extended $(p,q)$-Whittaker function by using the function $\Phi_{p,q}$ 
 and establish its various properties and associated formulas such as 
   integral representations, some transformation formulas and differential formulas.
    Relevant connections of the results presented here
with  those involving relatively simple Whittaker functions are also pointed out.

 \end{abstract}


\maketitle

\section{Introduction and preliminaries}\label{sec1}

We begin by recalling the classical beta function
$B(\alpha,\beta)$
 \begin{equation}\label{beta}
  B(\alpha,\, \beta) = \left\{ \aligned & \int_0^1 \, t^{\alpha -1} (1-t)^{\beta -1} \, dt
    \quad (\min\{\Re(\alpha), \, \Re(\beta)\}>0) \\
    &\frac{\Gamma (\alpha) \, \Gamma (\beta)}{\Gamma (\alpha+ \beta)}
   \hskip 23mm  \left(\alpha,\, \beta
      \in \mathbb{C}\setminus {\mathbb Z}_0^- \right), \endaligned \right.
\end{equation}
where $\Gamma$ denotes the familiar gamma function (see, e.g., \cite[Section 1.1]{Sr-Ch-12}).
Here and in the following, let $\mathbb{C}$, $\mathbb{R}^{+}$, $\mathbb{N}$, and $\mathbb{Z}_0^{-}$
be the sets of complex numbers, positive real numbers, positive integers, and non-positive integers,
respectively, and let $\mathbb{N}_0:=\mathbb{N}\cup \{0\}$ and $\mathbb{R}_0^{+}:= \mathbb{R}^{+}\cup \{0\}$.

The Gauss hypergeometric function ${}_2F_1$   and the confluent hypergeometric function ${}_1F_1$
are defined by (see, e.g.,  \cite{Rainville1960})
 \begin{eqnarray}\label{Hyper}
{}_2F_1(a,b;c;z)=\sum\limits_{n=0}^{\infty}\frac{(a)_n(b)_n}{(c)_n}\frac{z^n}{n!}
\quad \left(|z|<1;\,c  \in \mathbb{C}\setminus {\mathbb Z}_0^-  \right)
\end{eqnarray}
 and
 \begin{eqnarray}\label{Chyper}
{}_1F_1(a;c;z)=\sum\limits_{n=0}^{\infty}\frac{(a)_n}{(c)_n}\frac{z^n}{n!} \quad
   \left(|z|<\infty;\,c  \in \mathbb{C}\setminus {\mathbb Z}_0^-  \right),
\end{eqnarray}
where $(\lambda)_n$ is the Pochhammer symbol
defined (for $\lambda \in \mathbb C$) by (see  \cite[p.~2 and pp.~4-6]{Sr-Ch-12})
\begin{equation}\label{Poch} \aligned (\lambda)_n :
 & =\left\{\aligned & 1  \hskip 44 mm (n=0) \\
        & \lambda (\lambda +1) \cdots (\lambda+n-1) \quad (n \in {\mathbb N})
   \endaligned \right. \\
&= \frac{\Gamma (\lambda +n)}{ \Gamma (\lambda)}
      \quad \left(\lambda \in {\mathbb C} \setminus {\mathbb Z}_0^-\right).
\endaligned
\end{equation}
Integral representations for the Gauss hypergeometric function ${}_2F_1$   and the confluent hypergeometric function ${}_1F_1$
are recalled  (see, e.g.,  \cite[p. 65 and p. 70]{Sr-Ch-12})
\begin{equation}\label{Ihyper}
    _2F_1\,(a,\, b;\,c;\, z) = \frac{\Gamma (c)}{\Gamma (a)\,\Gamma (c-a)}
       \int_0^1 \, t^{a-1}\,(1-t)^{c-a-1}\, (1-zt)^{-b} \, dt \end{equation}
 \begin{equation*}
   (\Re(c) > \Re(a)>0; \,\, |\arg (1-z)| < \pi)
 \end{equation*}
 and
  \begin{equation}\label{Ichyper}
  \aligned
  _1F_1\,(a;\,c;\, z)= & \frac{\Gamma (c)}{\Gamma (a)\,\Gamma (c-a)}\,
    \int_0^1 \, t^{a-1}\, (1-t)^{c-a-1}\, e^{zt}\, dt \\
        & (\Re(c) > \Re(a) >0).
\endaligned
\end{equation}

Chaudhry et al. introduced the following  extended beta function (see \cite[Eq. (1.7)]{Chaudhry1997})
\begin{eqnarray}\label{Ebeta}
B(\alpha,\beta;p)=B_p(\alpha,\beta)=\int\limits_{0}^{1}\,t^{\alpha-1}\,(1-t)^{\beta-1}\,
e^{-\frac{p}{t(1-t)}}\,dt \quad (\Re(p)>0).
\end{eqnarray}
Obviously, $B(\alpha,\beta;0)=B(\alpha,\beta)$.   Chaudhry et al. \cite{Chaudhrya2004} introduced and investigated  the following extended  hypergeometric function $F_p$ and confluent hypergeometric function $\Phi_p$
\begin{eqnarray}\label{Ehyper}
F_p(a,b;c;z)=\sum\limits_{n=0}^{\infty}\frac{B_p(b+n, c-b)}{B(b, c-b)}(a)_n\frac{z^n}{n!}
\end{eqnarray}
\begin{equation*}
  \left(p \in \mathbb{R}_0^{+};\,|z|<1;\,\Re(c)>\Re(b)>0 \right)
\end{equation*}
and
\begin{eqnarray}\label{ECH}
\Phi_p(b;c;z)=\sum\limits_{n=0}^{\infty}\frac{B_p(b+n, c-b)}{B(b, c-b)}\frac{z^n}{n!} \quad  \left(p \in \mathbb{R}_0^{+};\,\Re(c)>\Re(b)>0 \right)
\end{eqnarray}
including their integral representations
\begin{eqnarray}\label{IEhyper}
F_p(a,b;c;z)=\frac{1}{B(b, c-b)}\,\int_0^1t^{b-1}(1-t)^{c-b-1}(1-zt)^{-a}\exp\Big[-\frac{p}{t(1-t)}\Big]\,dt,
\end{eqnarray}
\begin{equation*}
  \left(p \in \mathbb{R}^+;\, p=0 \,\,\,\text{and}\,\,\, |\arg(1-z)|<\pi;\, \Re(c)>\Re(b)>0 \right)
\end{equation*}
and
\begin{eqnarray}\label{IEChyper}
\Phi_p(b;c;z)=\frac{1}{B(b, c-b)}\int_0^1\,t^{b-1}(1-t)^{c-b-1}\exp\Big(zt-\frac{p}{t(1-t)}\Big)\,dt
\end{eqnarray}
\begin{equation*}
  \left(p \in \mathbb{R}^+;\, p=0 \,\,\,\text{and}\,\,\,  \Re(c)>\Re(b)>0 \right).
\end{equation*}

Clearly, the particular cases of \eqref{Ehyper}, \eqref{ECH}, \eqref{IEhyper}, and \eqref{IEChyper}
when $p=0$ reduce, respectively, to \eqref{Hyper}, \eqref{Chyper}, \eqref{Ihyper}, and \eqref{Ichyper}.

 Choi et al. \cite{Choi2014}  defined the following extension of the beta function
\begin{eqnarray}\label{EEbeta}
B(x,y;p,q)=B_{p,q}(x,y)=\int\limits_{0}^{1}\,t^{x-1}\,(1-t)^{y-1}\,\exp \left[-\frac{p}{t}-\frac{q}{1-t}\right]\,dt
\end{eqnarray}
\begin{equation*}
  \left(\min\{\Re(p),\, \Re(q) \}>0\right).
\end{equation*}
Using \eqref{EEbeta}, they \cite{Choi2014} defined the following $(p,q)$-hypergeometric function and
 $(p,q)$-confluent hypergeometric function, respectively, by
 \begin{eqnarray}\label{FEhyper}
F_{p,q}(a,b;c;z)=\sum\limits_{n=0}^{\infty}\frac{B_{p,q}(b+n, c-b)}{B(b, c-b)}(a)_n\,\frac{z^n}{n!}
\end{eqnarray}
\begin{equation*}
  \left(p,\,q \in \mathbb{R}_0^+;\, \Re(c)>\Re(b)>0 \right)
\end{equation*}
 and
\begin{eqnarray}\label{FECH}
\Phi_{p,q}(b;c;z)=\sum\limits_{n=0}^{\infty}\frac{B_{p,q}(b+n, c-b)}{B(b, c-b)}\,\frac{z^n}{n!}
\end{eqnarray}
\begin{equation*}
  \left(p,\,q \in \mathbb{R}_0^+;\, \Re(c)>\Re(b)>0 \right)
\end{equation*}
and presented their integral representations
\begin{eqnarray}\label{FEhyperInt}
\aligned
 F_{p,q}(a,b;c;z)=\frac{1}{B(b,c-b)}\,\int_0^1\,& t^{b-1}(1-t)^{c-b-1}(1-zt)^{-a}\\
& \times \exp\Big(-\frac{p}{t}-\frac{q}{1-t}\Big)\,dt
\endaligned
\end{eqnarray}
\begin{equation*}
  \left(p,\,q \in \mathbb{R}_0^+;\,\Re(c)>\Re(b)>0;\, |\arg(1-z)|<\pi   \right)
\end{equation*}
and
\begin{eqnarray}\label{FEChyper}
\aligned
\Phi_{p,q}(b;c;z)=\frac{1}{B(b, c-b)}\,\int_0^1\,& t^{b-1}(1-t)^{c-b-1}\\
& \times \exp\Big(zt-\frac{p}{t}-\frac{q}{1-t}\Big)\,dt
\endaligned
\end{eqnarray}
\begin{equation*}
  \left(p,\,q \in \mathbb{R}_0^+;\,\Re(c)>\Re(b)>0  \right).
\end{equation*}
 Obviously, when $p=q=0$, \eqref{FEhyper}-\eqref{FEChyper}  reduce, respectively, to
 \eqref{Ehyper}-\eqref{IEChyper}.
 They  \cite[Eq. (11.4)]{Choi2014} also obtained the following transformation formula
\begin{eqnarray}\label{trans}
\Phi_{p,q}(b;c;z)=e^z\,\Phi_{q,p}\Big(c-b;c;-z\Big).
\end{eqnarray}

Whittaker and Watson \cite[Chapter XVI]{Whittaker} used the confluent hypergeometric function ${}_1F_1$ to define the Whittaker function
\begin{eqnarray}\label{Whittaker}
M_{\lambda,\rho}(z)=z^{\rho+\frac{1}{2}}\,\exp\Big(\hskip -2mm -\frac{z}{2}\Big)\,{}_1F_1\Big(\rho-\lambda+\frac{1}{2};2\rho+1;z\Big)
\end{eqnarray}
\begin{equation*}
  \left(\Re(\rho)>-\frac{1}{2};\,  \Re(\rho\pm\lambda)>-\frac{1}{2} ;\, z \in \mathbb{C}\setminus (-\infty,0]\right).
\end{equation*}
 This Whittaker function \eqref{Whittaker}  is a special
solution of the Whittaker's differential equation and a modified form of the confluent hypergeometric
equation  to make formulas involving the solutions
more symmetric (see \cite{Whittaker1}).

Nagar et al. \cite{Nagar} used the extended confluent hypergeometric series $\Phi_p$ in \eqref{ECH} to  extend the Whittaker function $M_{\lambda,\rho}(z)$ as follows:
\begin{eqnarray}\label{EWhittaker}
M_{p;\lambda,\rho}(z)=z^{\rho+\frac{1}{2}}\exp\Big(\hskip -2mm -\frac{z}{2}\Big)\Phi_p\Big(\rho-\lambda+\frac{1}{2};2\rho+1;z\Big)
\end{eqnarray}
\begin{equation*}
  \left(p \in \mathbb{R}_0^+;\, \Re(\rho)>-\frac{1}{2};\,  \Re(\rho\pm\lambda)>-\frac{1}{2};\, z \in \mathbb{C}\setminus (-\infty,0] \right).
\end{equation*}

\vskip 3mm
Here, by using  the $(p,q)$-confluent hypergeometric function  $\Phi_{p,q}$ in \eqref{FECH},
we define the following $(p,q)$-Whittaker function
\begin{eqnarray}\label{FWhittaker}
M_{p,q;\lambda,\rho}(z)=z^{\rho+\frac{1}{2}}\exp\Big(\hskip -2mm -\frac{z}{2}\Big)\Phi_{p,q}\Big(\rho-\lambda+\frac{1}{2};2\rho+1;z\Big)
\end{eqnarray}
\begin{equation*}
  \left(p,\,q \in \mathbb{R}_0^+;\, \Re(\rho)>-\frac{1}{2};\,  \Re(\rho\pm\lambda)>-\frac{1}{2};\, z \in \mathbb{C}\setminus (-\infty,0]  \right).
\end{equation*}
Clearly, $M_{p,p;\lambda,\rho}(z)=M_{p;\lambda,\rho}(z)$ in \eqref{EWhittaker}. Then we investigate   a number of formulas involving
the  $(p,q)$-Whittaker function in \eqref{FWhittaker}, systematically, such as various integral representations,
a transformation formula, a Mellin transform, and a derivative formula.

\section{Formulas involving the $(p,q)$-Whittaker function} \label{sec2}

Here, we present a number of formulas involving the $(p,q)$-Whittaker function in \eqref{FWhittaker},
in a rather systematic way.

\vskip 3mm

\begin{theorem}\label{Th1}
Let $\Re(\rho)>\Re(\rho\pm\lambda)>-\frac{1}{2}$ and $z \in \mathbb{C}\setminus (-\infty,0]$. Also, let $p,\, q  \in \mathbb{R}_0^+$.
 Then the following integral representations hold true:
\begin{equation}\label{int1}
\aligned
M_{p,q;\lambda,\rho}(z)=& \frac{z^{\rho+\frac{1}{2}}\exp(-\frac{z}{2})}{B(\rho-\lambda+\frac{1}{2},\rho+\lambda+\frac{1}{2})}\\
 & \times \int_0^1 t^{\rho-\lambda-\frac{1}{2}}(1-t)^{\rho+\lambda-\frac{1}{2}}\exp\Big(zt-\frac{p}{t}-\frac{q}{1-t}\Big)\,dt;
 \endaligned
\end{equation}

\begin{equation}\label{int2}
\aligned
M_{p,q;\lambda,\rho}(z)=& \frac{z^{\rho+\frac{1}{2}}\exp(\frac{z}{2})}{B(\rho-\lambda+\frac{1}{2},\rho+\lambda+\frac{1}{2})}\\
 & \times \int_0^1 u^{\rho+\lambda-\frac{1}{2}}(1-u)^{\rho-\lambda-\frac{1}{2}}\exp\Big(-zu-\frac{p}{1-u}-\frac{q}{u}\Big)du;
 \endaligned
\end{equation}

\begin{equation}\label{int3}
\aligned
M_{p,q;\lambda,\rho}(z)=&\frac{(b-a)^{-2\rho}z^{\rho+\frac{1}{2}}
\exp(-\frac{z}{2})}{B(\rho-\lambda+\frac{1}{2},\rho+\lambda+\frac{1}{2})}\int_a^b (u-a)^{\rho-\lambda-\frac{1}{2}}(b-u)^{\rho+\lambda-\frac{1}{2}}\\
&\hskip 10mm\times\exp\Big[\frac{z(u-a)}{b-a}-\frac{p(b-a)}{u-a}-\frac{q(b-a)}{b-u}\Big]\,du
\endaligned
\end{equation}
\begin{equation*}
  \left(a,\, b \in \mathbb{R} \,\,\, \text{with}\,\,\, b>a   \right);
\end{equation*}

\begin{equation}\label{int4}
\aligned
M_{p,q;\lambda,\rho}(z)= & \frac{z^{\rho+\frac{1}{2}}
\exp(-\frac{z}{2})}{B(\rho-\lambda+\frac{1}{2},\rho+\lambda+\frac{1}{2})}\int_0^\infty u^{\rho-\lambda-\frac{1}{2}}(1+u)^{-(2\rho+1)}\\
&\hskip 10mm  \times\exp\Big[\frac{zu}{1+u}-\frac{p(1+u)}{u}-q(1+u)\Big]\,du;
\endaligned
\end{equation}

\begin{equation}\label{int5}
\aligned
M_{p,q;\lambda,\rho}(z)=&\frac{2^{-2\rho}z^{\rho+\frac{1}{2}}
}{B(\rho-\lambda+\frac{1}{2},\rho+\lambda+\frac{1}{2})}\int_{-1}^1\, (1+u)^{\rho-\lambda-\frac{1}{2}}(1-u)^{\rho+\lambda-\frac{1}{2}}\\
&\hskip 30mm \times\exp\Big[\frac{z(u+1)}{2}-\frac{2p}{1+u}-\frac{2q}{1-u}\Big]\, du.
\endaligned
\end{equation}
\end{theorem}

\begin{proof}
Applying the integral representation of the $(p,q)$-confluent hypergeometric function in \eqref{FEChyper}
to the $(p,q)$-Whittaker function in \eqref{FWhittaker}, we obtain the integral representation \eqref{int1}.
Then, by setting $t=1-u$, $t=\frac{u-a}{b-a}$, and $t=\frac{u}{1+u}$ in (\ref{int1}), we get \eqref{int2},
\eqref{int3}, and \eqref{int4}, respectively. Finally, setting $a=-1$ and $b=1$ in \eqref{int3} yields \eqref{int5}.
\end{proof}

\vskip 3mm
\begin{remark}\label{rmk1}
In view of \eqref{FEChyper}, we find from \eqref{int2} that
the $(p,q)$-Whittaker function  can also be expressed in the following form
\begin{eqnarray}
M_{p,q;\lambda,\rho}(z)=z^{\rho+\frac{1}{2}}\exp\left(\frac{z}{2}\right)\,\Phi_{q,p}\Big(\rho+\lambda+\frac{1}{2};2\rho+1;-z\Big)
\end{eqnarray}
\begin{equation*}
  \left(p,\,q \in \mathbb{R}_0^+;\, \Re(\rho)>-\frac{1}{2};\,  \Re(\rho\pm\lambda)>-\frac{1}{2};\, z \in \mathbb{C}\setminus (-\infty,0]  \right).
\end{equation*}
\end{remark}

\vskip 3mm
\begin{theorem}\label{thm2}
The following transformation formula for the $(p,q)$-Whittaker function holds true:
\begin{eqnarray}
M_{p,q;\lambda,\rho}(z)= (-1)^{\rho + \frac{1}{2}}\,M_{q,p;-\lambda,\rho}(-
z)
\end{eqnarray}
\begin{equation*}
  \left(p,\,q \in \mathbb{R}_0^+;\, \Re(\rho)>-\frac{1}{2};\,  \Re(\rho\pm\lambda)>-\frac{1}{2};\, z \in \mathbb{C}\setminus (-\infty,0]  \right).
\end{equation*}
\end{theorem}

\begin{proof}
Applying \eqref{trans} to \eqref{FWhittaker}, we get
\begin{equation}\label{thm2-pf1}
 M_{p,q;\lambda,\rho}(z)= z^{\rho+\frac{1}{2}}\,\exp\left(\frac{z}{2}\right)\, \Phi_{q,p}\Big(\rho+\lambda+\frac{1}{2};2\rho+1;-z\Big).
\end{equation}
Using \eqref{FWhittaker} in \eqref{thm2-pf1}, we obtain the desired result.
\end{proof}

\vskip 3mm

\begin{theorem}\label{Th2}
The following Mellin transformation  for the $(p,q)$-Whittaker function holds true:
\begin{equation}\label{Th2-eq1}
\aligned
&\mathfrak{M}\{M_{p,q;\lambda,\rho}(z);p\rightarrow r,q\rightarrow s\} \\ &\hskip 10mm =\frac{z^{\rho+\frac{1}{2}}\exp(-\frac{z}{2})\,\Gamma(r)\,\Gamma(s)\, B(\rho+r-\lambda+\frac{1}{2},\rho+s+\lambda+\frac{1}{2})}
{B(\rho-\lambda+\frac{1}{2},\rho+\lambda+\frac{1}{2})}\\
&\hskip 15mm \times {}_1F_1\Big(\rho+r-\lambda+\frac{1}{2};2\rho+r+s+1;z\Big)
\endaligned
\end{equation}
\begin{equation*}
\left(\min\{\Re(s),\, \Re(r) \}>0;\, \Re(\rho\pm\lambda+r)>-\frac{1}{2}  \right).
\end{equation*}

\end{theorem}

\begin{proof}
Let $\mathcal{L}_1$ be the left side of \eqref{Th2-eq1}. By definition of Mellin transformation, we have
\begin{equation}\label{Th2-pf1}
 \mathcal{L}_1 =\int_0^\infty\int_0^\infty p^{r-1}q^{s-1}M_{p,q;\lambda,\rho}(z)\,dp\,dq.
\end{equation}
Replacing $M_{p,q;\lambda,\rho}(z)$ in \eqref{Th2-pf1} by its integral representation \eqref{int1}
and interchanging the order of the integrals, which is verified under the given conditions here,
we obtain
\begin{equation}\label{Th2-pf2}
\aligned
   \mathcal{L}_1=&\int_0^\infty\int_0^\infty p^{r-1}q^{s-1}M_{p,q,\lambda,\rho}(z)\,dp\,dq\\
=& \frac{z^{\rho+\frac{1}{2}}\exp(-\frac{z}{2})}{B(\rho-\lambda+\frac{1}{2},\rho+\lambda+\frac{1}{2})}\int_0^1t^{\rho-\lambda-\frac{1}{2}}
(1-t)^{\rho+\lambda-\frac{1}{2}}e^{zt}\\
&\hskip 5mm  \times\int_0^\infty p^{r-1}\exp\left(-\frac{p}{t}\right)\,dp\,\int_0^\infty q^{s-1}\exp\left(-\frac{q}{1-t}\right)\,dq\,dt.
\endaligned
\end{equation}

Using the following easily derivable formula
 \begin{equation}\label{Th2-pf3}
   \int_0^\infty\,u^{x-1}\,\exp (-\alpha u)\,du = \frac{\Gamma (x)}{\alpha^x} \quad \left(\Re(x)>0;\,\,
     \alpha \in \mathbb{R}^+ \right),
 \end{equation}
we get
\begin{equation}\label{Th2-pf4}
  \int_0^\infty p^{r-1}\exp\left(-\frac{p}{t}\right)dp\,
     \int_0^\infty q^{s-1}\exp\left(-\frac{q}{1-t}\right)dq=t^r(1-t)^s\Gamma(r)\Gamma(s).
\end{equation}

Setting \eqref{Th2-pf4} in \eqref{Th2-pf2},  with the aid of \eqref{Ichyper}, we obtain the desired result.

\end{proof}

\vskip 3mm

\begin{theorem}\label{Th3}
The following integral formula  involving the $(p,q)$-Whittaker function  holds true:
\begin{equation}\label{Th3-eq1}
\aligned
& \int_0^\infty z^{\delta-1}e^{-\alpha z}M_{p,q;\lambda,\rho}(\mu z)\,dz=\frac{\mu^{\rho+\frac{1}{2}}\Gamma(\delta+\rho+\frac{1}{2})}{(\alpha+\frac{\mu}{2})^{\delta+\rho+\frac{1}{2}}}\\
&\hskip 30mm \times F_{p,q}\Big(\delta+\rho+\frac{1}{2}, \rho-\lambda+\frac{1}{2}; 2\rho+1;\frac{2\mu}{2\alpha+\mu}\Big)
\endaligned
\end{equation}
\begin{equation*}
  \left(p,\, q \in \mathbb{R}_0^+;\,\alpha+\frac{\mu}{2}>0,\, \mu<0,\, 2\alpha+\mu>2|\mu|;\,\Re(\delta+\rho)>-\frac{1}{2},\,  \Re(\rho\pm\lambda)>-\frac{1}{2}  \right).
\end{equation*}
\end{theorem}

\begin{proof}
Let $\mathcal{L}_2$ be the left side of \eqref{Th3-eq1}. By using the integral representation of $M_{p,q;\lambda,\rho}(z)$
  in \eqref{int1} and interchanging the order of integrals, which is verified under the given assumptions here, we have
  \begin{equation}\label{Th3-pf1}
   \aligned
   \mathcal{L}_2 =  \frac{\mu^{\rho +\frac{1}{2}}}{B(\rho-\lambda+\frac{1}{2},\rho+\lambda+\frac{1}{2})}\,
       & \int_0^1\,t^{\rho-\lambda-\frac{1}{2}}(1-t)^{\rho+\lambda-\frac{1}{2}}\exp\Big(-\frac{p}{t}-\frac{q}{1-t}\Big)\\
     & \times \left\{\int_{0}^{\infty}\,z^{\delta+\rho-\frac{1}{2}}\,e^{-(\alpha+\frac{\mu}{2}-\mu t )z}\,dz\right\}\,dt.
    \endaligned
  \end{equation}
As in \eqref{Th2-pf3}, we find
\begin{equation}\label{Th3-pf2}
 \int_{0}^{\infty}\,z^{\delta+\rho-\frac{1}{2}}\,e^{-(\alpha+\frac{\mu}{2}-\mu t )z}\,dz
   = \frac{\Gamma \left(\delta + \rho + \frac{1}{2}\right)}{\left(\alpha + \frac{\mu}{2}-\mu t\right)^{\delta + \rho + \frac{1}{2}}}
\end{equation}
\begin{equation*}
  \left(\Re(\delta+\rho)>-\frac{1}{2};\,\alpha+\frac{\mu}{2},\, -\mu \in \mathbb{R}^+   \right).
\end{equation*}
Setting \eqref{Th3-pf2} in \eqref{Th3-pf1}, we obtain
\begin{equation}\label{Th3-pf3}
   \aligned
  \mathcal{L}_2 =&  \frac{\Gamma \left(\delta + \rho + \frac{1}{2}\right)\,\mu^{\rho +\frac{1}{2}}}{ \left(\alpha +\frac{\mu}{2}\right)^{\delta + \rho + \frac{1}{2}}\,B(\rho-\lambda+\frac{1}{2},\rho+\lambda+\frac{1}{2})}\\
       &\times \int_0^1\,t^{\rho-\lambda-\frac{1}{2}}(1-t)^{\rho+\lambda-\frac{1}{2}}\exp\Big(-\frac{p}{t}-\frac{q}{1-t}\Big)\,
        \left(1- \frac{2\mu}{2\alpha+\mu}t\right)^{-\left(\delta + \rho + \frac{1}{2}\right)}\, dt.
    \endaligned
  \end{equation}
By the generalized binomial theorem, we get
\begin{equation}\label{Th3-pf4}
    \left(1- \frac{2\mu}{2\alpha+\mu}t\right)^{-\left(\delta + \rho + \frac{1}{2}\right)}
    =\sum_{n=0}^{\infty}\,\frac{\big(\delta + \rho + 1/2\big)_n}{n!}\,  \left(\frac{2\mu}{2\alpha+\mu}\right)^n\, t^n\,\,\,
         \end{equation}
\begin{equation*}
  \left( \left|2\mu\,t|/|2\alpha+\mu t\right|<1 \right).
\end{equation*}
Using  \eqref{Th3-pf4} in \eqref{Th3-pf3} and interchanging the order of summation and integral,
which is verified under the assumptions here, and using \eqref{FEhyperInt},  we have
\begin{equation}\label{Th3-eq5}
\aligned
 \mathcal{L}_2 =&\frac{\mu^{\rho+\frac{1}{2}}\Gamma(\delta+\rho+\frac{1}{2})}{(\alpha+\frac{\mu}{2})^{\delta+\rho+\frac{1}{2}}}\\
& \times \sum_{n=0}^{\infty}\,\frac{B_{p,q}(\rho-\lambda+\frac{1}{2}+n, \rho+\lambda+ \frac{1}{2})}{B(\rho-\lambda+\frac{1}{2}, \rho+\lambda+ \frac{1}{2})}\,\frac{(\delta+\rho+1/2)_n}{n!}\,\left(\frac{2\mu}{2\alpha+\mu}\right)^n,
\endaligned
\end{equation}
which,  in view of \eqref{FEhyper}, is equal to the right side of \eqref{Th3-eq1}.

\end{proof}

\vskip 3mm

\begin{theorem}\label{Th5}
The following derivative formula holds true:
\begin{eqnarray}\label{deriv}
\frac{d^n}{dz^n}\Big\{e^{\frac{z}{2}}z^{-\rho-\frac{1}{2}}M_{p,q;\lambda,\rho}(z)\Big\} =\frac{(\rho-\lambda+\frac{1}{2})_n}{(2\rho+1)_n}e^{\frac{z}{2}}z^{-\rho-\frac{n}{2}-\frac{1}{2}}M_{p,q;\lambda-\frac{n}{2},\rho+\frac{n}{2}}(z)
 \quad \left(n\in\mathbb{N}_0\right).
\end{eqnarray}
\end{theorem}

\begin{proof}
From \eqref{FWhittaker}, we have
\begin{eqnarray}\label{deriv1}
e^{\frac{z}{2}}z^{-\rho-\frac{1}{2}}M_{p,q;\lambda,\rho}(z)=\Phi_{p,q}\Big(\rho-\lambda+\frac{1}{2};2\rho+1;z\Big).
\end{eqnarray}
Recall the following formula (see \cite[Eq. (9.3)]{Choi2014}
\begin{eqnarray}\label{deriv2}
\frac{d^n}{dz^n}\Big\{\Phi_{p,q}\Big(b;c;z\Big)\}=\frac{(b)_n}{(c)_n}\Phi_{p,q}\Big(b+n;c+n;z\Big)
\quad \left(n\in\mathbb{N}_0\right).
\end{eqnarray}
Applying \eqref{deriv2} to \eqref{deriv1}, we get the desired result.
\end{proof}

\vskip 3mm

\section{Remarks and special cases}

In this paper, we made a main use of the results in \cite{Choi2014} to give certain formulas 
involving the $(p,q)$-Whittaker function $M_{p,q;\lambda,\rho}(z)$ in \eqref{FWhittaker}, whose essential factor is
the function $\Phi_{p,q}(b;c;z)$. Some important properties  for the $\Phi_{p,q}(b;c;z)$ and various formulas involving
it have already been established in \cite{Choi2014}. Yet, for   a more convenient and faster use of 
certain properties for the $(p,q)$-Whittaker function $M_{p,q;\lambda,\rho}(z)$ and diverse formulas involving it,
we intend to write this paper.

The results presented here, being very general, can be reduced to yield those involving relatively simple 
Whittaker functions. The results given here when $p=q$
reduce to the known results associated with the extended Whittaker function $M_{p;\lambda,\rho}(z)$ in \eqref{EWhittaker}
(see \cite{Nagar}). The special cases $p=q=0$ of the results given here will yield the corresponding ones
involving the classical Whittaker function $M_{\lambda,\rho}(z)$ in \eqref{Whittaker}.
 
 Among numerous special cases of the results presented here, we choose to demonstrate only one formula.
 Setting $\delta=1$ and $\mu=-1$  in \eqref{Th3-eq1}, we obtain the Laplace transformation of 
 the $(p,q)$-Whittaker function $M_{p,q;\lambda,\rho}(-t)$ in \eqref{FWhittaker}
 \begin{equation} \label{Laplace-M-p-q}
\aligned
& \int_0^\infty \,e^{-\alpha t}M_{p,q;\lambda,\rho}(-t)\,dt=\frac{(-1)^{\rho+\frac{1}{2}}\Gamma(\rho+\frac{3}{2})}{(\alpha-\frac{1}{2})^{\rho+\frac{3}{2}}}\\
&\hskip 20mm \times F_{p,q}\Big(\rho+\frac{3}{2}, \rho-\lambda+\frac{1}{2}; 2\rho+1;\frac{2}{1-2\alpha}\Big)
\endaligned
\end{equation}
\begin{equation*}
  \left(p,\, q \in \mathbb{R}_0^+;\,\alpha>\frac{3}{2};\,\Re(\rho)>-\frac{3}{2},\,  \Re(\rho\pm\lambda)>-\frac{1}{2};\, \arg(-1)=\pi  \right).
\end{equation*}

\end{document}